\documentclass{amsart}
\usepackage[utf8]{inputenc}
\usepackage{amsfonts,amssymb,amsmath,amsthm,stmaryrd}
\usepackage[usenames]{xcolor}
\usepackage[e]{esvect}

\newtheorem{theorem}{Theorem}
\newtheorem{lemma}{Lemma}
\newtheorem{corollary}{Corollary}

\theoremstyle{definition}
\newtheorem{remark}{Remark}

\definecolor{uugreen}{cmyk}{1,0,0.75,0}
\definecolor{uured}{cmyk}{0.2,1,0.9,0.1}
\definecolor{uublue}  {cmyk}{0.9,0.55,0,0}
\newcommand{\wear}{\ensuremath{{\sf S}^1_2}}

\newcommand{\fpnote}[1]{\marginpar{\footnotesize \textcolor{uured}{#1}}}
 \newcommand{\nrhd}{\mathrel{\not\! \rhd}}

\newcommand{\dijosu}{\sqcup}

\newcommand{\sopre}{{\sf S}}
\newcommand{\woso}{{\mathfrak w}}
\newcommand{\doso}{{\mathfrak d}}

\newcommand{\obj}{\ensuremath{\mathfrak o}}
\newcommand{\tuso}{{\mathfrak t}}
\newcommand{\clso}[1]{{\mathfrak c}_{#1}}

\renewcommand{\iff}{\leftrightarrow}
\renewcommand{\vec}[1]{\vv #1}

\title{Finitely Axiomatized Theories lack Self-Comprehension}
\author{Fedor Pakhomov}
\thanks{Research of Fedor Pakhomov is supported by FWO Senior Postdoctoral Fellowship, project 1283021N.}
 \address{Vakgroep Wiskunde: Analysis, Logic and Discrete Mathematics, 
Ghent University,
Krijgslaan 281,
B9000~~Ghent,
Belgium\newline
and Steklov Mathematical Institute of Russian Academy of Sciences, 
Gubkina 8,
119991 Moscow, Russia}
\email{fedor.pakhomov@ugent.be}

\author{Albert Visser}
\thanks{We thank Lev Beklemishev for his helpful comments and encouragement.}
 \address{Philosophy, Faculty of Humanities,
                Utrecht University,
               Janskerkhof 13,
                3512BL~~Utrecht, The Netherlands}
\email{a.visser@uu.nl}

\date{September 2021}

\begin{document}

\maketitle
    

\begin{abstract}
   In this paper we prove that no consistent finitely axiomatized theory one-dimensionally interprets its own extension with predicative comprehension. This constitutes a result with the flavor of the Second Incompleteness Theorem whose formulation is completely arithmetic-free. Probably the most important novel feature that distinguishes our result from the previous results of this kind is that it is applicable to arbitrary weak theories, rather than to extensions of some base theory.
    
    
    The methods used in the proof of the main result yield a new perspective on the notion of sequential theory, in the setting of forcing-interpretations.
    
\end{abstract}

\section{Introduction}
In this paper we provide an impossibility argument in the niche of the Second Incompleteness Theorem.
We show that no consistent finitely axiomatized theory $T$ can one-dimensionally interpret its own extension, 
${\sf PC}(T)$ that is the second-order extension of $T$ by the predicative comprehension principle
\[
\exists X\,\forall y\,(y\in X\mathrel{\leftrightarrow}\varphi(y)),\]
where $\varphi$ has no second-order quantifiers and $X\not\in\mathsf{FV}(\varphi)$.

Our result is inspired by well-known results about the connection between Predicative Comprehension
and consistency in the case of sequential theories, roughly, theories with sufficient coding machinery.
Two salient results are that Peano Arithmetic, {\sf PA}, does not interpret ${\sf ACA}_0$ and that Zermelo-Fraenkel Set Theory,
{\sf ZF}, does not interpret G\"odel-Bernays Set Theory, {\sf GB}.
One way of proving these results employs
the fact that ${\sf ACA}_0$ interprets $\wear+{\sf Con}({\sf PA})$ and
${\sf GB}$ interprets $\wear+{\sf Con}({\sf ZF})$. 
In fact, the relationship is even tighter, 
${\sf ACA}_0$ is mutually interpretable with $\wear+{\sf Con}({\sf PA})$
and ${\sf GB}$ is mutually interpretable with $\wear+{\sf Con}({\sf ZF})$.
Here \wear\ is Buss's weak arithmetic.
More generally, we have the following result. Suppose $U$ is a sequential theory that
is axiomatized by a scheme $\Theta$. Let ${\sf PC}^{\sf schem}(\Theta)$ be the theory obtained
by taking Predicative Comprehension over the signature of $U$ and adding the universally
quantified version of $\Theta$, where the schematic variables are replaced by class variables.
We have: ${\sf PC}^{\sf schem}(\Theta)$ is mutually interpretable with $\wear + {\sf Con}(\Theta)$.
(See \cite{viss:seco11}, for more information.)
In combination with an appropriate version of the
Second Incompleteness Theorem, we find (\dag) $U$ does not interpret ${\sf PC}^{\sf schem}(\Theta)$.

In our paper we study (\dag) outside of its comfort zone of sequential theories.
We restrict ourselves to finitely axiomatized theories and to one-dimensional interpretability.
However, under these restrictions we prove the result for all theories.
We note that in the finitely axiomatized case, we only need
${\sf PC}(U)$, the result of simply adding Predicative Comprehension to $U$, in stead of the
more fancy ${\sf PC}^{\sf schem}(U)$.

We note that the result that $T$ does not one-dimensionally interpret ${\sf PC}(T)$
shows that $T$ does not interpret $\wear+{\sf Con}(T)$, since $\wear+{\sf Con}(T)$
interprets ${\sf PC}(T)$. The argument for the interpretability of ${\sf PC}(T)$ in
$\wear+{\sf Con}(T)$ is essentially a refinement of the proof of the Completeness Theorem
and does not involve diagonalization. Thus, for a restricted class of cases, our result
implies a version of the Second Incompleteness Theorem.

Our paper provides some spin-offs that hold independent interest.
We present these results in Section~\ref{forcefulsmurf}.

A first result tells us that the extension of a theory $T$ with adjunctive
sets 
is mutually forcing-interpretable with the extension of $T$ with the adjunctive theory of binary relation classes plus the no-universe axiom. The result has the extra feature that the forcing-interpretations back-and-forth preserve the objects and relations of $T$. 
We note that adding adjunctive sets is a form of sequential closure, i.e., a way of making a
theory sequential.

A second result tells us that, if $T$ is finitely axiomatized and one-dimensionally interprets
$T$ on a provably smaller domain, then the extension of $T$ with $n$-ary adjunctive classes,
for sufficiently large $n$, forcing-interprets the extension of $T$ with adjunctive sets.

Thirdly, we show that, if $T$ is finitely axiomatized and one-dimensionally interprets
the extension of $T$ with adjunctive classes, then $T$ forcing-interprets the extension of
$T$ with adjunctive sets.

\subsection*{Genesis of this Work}
The questions leading to the results of this paper come from earlier work by Albert Visser.
The strengthening of the non-interpretability result of ${\sf PC}(T)$ in $T$ for
the sequential, finitely axiomatized case, to the case of pairing theories was discovered some time in a conversation of Albert Visser and Fedor Pakhomov. The basic proof strategy for Theorem~\ref{one_dim_thm} was discovered by Fedor Pakhomov. 

\section{Preliminaries}

All theories that we consider are one-sorted theories with equality and finite relational signature. We assume that the connectives in the first-order language are $\forall,\land$, and $\lnot$. We express all the other connectives using these ones. 

However, we frequently will consider theories that naturally should be considered $n$-sorted theories (with relational signature). In order to do this, we will identify an $n$-sorted theory $T$, whose sorts are $\sigma_1,\ldots,\sigma_n$
with the following one-sorted theory $T^\flat$. The signature of $T^\flat$ contains, in addition to the signature of $T$, unary predicate symbols $\sopre_{\sigma_1},\ldots,\sopre_{\sigma_n}$.
We consider the sorted quantifier $\forall x^{\sigma_i}\varphi$ to be a shorthand for $\forall x\,(\sopre_{\sigma_i}(x)\to\varphi)$. In addition to the explicitly given axioms of $T$, we have the following axioms:
\begin{enumerate}
    \item $\bigvee\limits_{1\le i\le n}\sopre_{\sigma_i}(x)$;
    \item $\lnot\, 
    (\sopre_{\sigma_i}(x)\land \sopre_{\sigma_j}(x))$, for $i<j$;
    \item $\exists x \;\sopre_{\sigma_i}(x)$, for each $i$;
    \item $R(x_1,\ldots,x_m)\to 
    (\sopre_{\sigma_{k_1}}(x_1)
    \land\ldots\land \sopre_{\sigma_{k_m}}(x_m))$, for each original $m$-ary predicate symbol $R$, whose $i$-th argument is of the sort $\sigma_{i}$, for $i<m$. 
    
    Here we treat identity separately: identity of each sort is simply the restriction of identity for the whole domain of $T^\flat$ to
    each of the domains $\sopre_\sigma$.
\end{enumerate}

For theories $T$ and $U$, we denote as $T\dijosu U$ the two-sorted theory that has all predicates of $T$ on the first sort, all predicates of $U$ on the second sort, and whose axioms are all the axioms of $T$ relativized to the first sort and all the axioms of $U$ relativized to the second sort.


We define theory $\mathsf{PC}_{\le n}(T)$ (Predicative Comprehension up to the arity $n$), for any theory $T$. This is 
the $n+1$ sorted theory, whose sorts are \obj\ and $\clso{1},\ldots,\clso{n}$. The predicates of $\mathsf{PC}_{\le n}(T)$ are the predicates of $T$ restricted to the sort \obj\ as well as the predicates $\langle x_1^{\obj},\ldots,x_i^{\obj}\rangle\in X^{(i)}$, for $1\le i\le n$. The axioms of $\mathsf{PC}_{\le n}(T)$ are as follows.
\begin{enumerate}
    \item The axioms of $T$ relativized to the sort \obj.
    \item $\exists X^{\clso{k}}\,
    \forall x_1^{\obj},\ldots, x_k^{\obj}\,
    ( \langle x_1,\ldots,x_k\rangle\in X\mathrel{\leftrightarrow} \varphi(x_1,\ldots,x_k))$, where all quantifiers in $\varphi$ are on the sort $\obj$ and there are no free occurrences of $X$.
    \item
    $\forall X^{\clso{k}}, Y^{\clso{k}}\,(\forall
    x_1^{\obj},\ldots, x_k^{\obj}\; (
    \langle x_1,\ldots, x_k\rangle\in X \leftrightarrow
    \langle x_1,\ldots,x_k\rangle\in Y)\to X=Y)$
\end{enumerate}
The theory $\mathsf{PC}(T)$ is $\mathsf{PC}_{\le 1}(T)$.

In this paper we consider multi-dimensional relative interpretations with parameters and definable equality.

Our main theorem is 
\begin{theorem}\label{one_dim_thm}No consistent finitely axiomatized theory $T$ can 1-dimensionally interpret $\mathsf{PC}(T)$ In other words,
for every consistent finitely axiomatized theory $T$ we have
$T\nrhd_1 {\sf PC}(T)$.
\end{theorem}

\section{Predicative comprehension and tuples}
\label{PC_and_tuples}

We write $T\rhd_m U$ if $T$ interprets $U$ by an $m$-dimensional interpretation.
We have the following trivial lemma:
\begin{lemma}\label{PC_funct}
If $T\rhd_1 U$, then $\mathsf{PC}(T)\rhd_1\mathsf{PC}(U)$.
\end{lemma}
And its multi-dimensional generalization:
\begin{lemma}\label{PC_funct_mult}
If $T\rhd_n U$, then $\mathsf{PC}_{\le nm}(T)\rhd_n \mathsf{PC}_{\le m}(U)$.
\end{lemma}

It is sometimes pleasant to treat dimension using an auxiliary theory
that adds $i$-tuples for $2\leq i \leq n$ to the given base theory.
Let $\mathsf{Tuple}_{\le n}(T)$ be the following $n$-sorted theory. The sorts of $\mathsf{Tuple}_{\le n}(T)$ are $\tuso_1,\ldots,\tuso_n$. Here
$\tuso_1$ may be identified with $\obj$, the sort of basic
objects. The signature of $\mathsf{Tuple}_{\le n}$ consists of all the predicates of $T$ on the sort $\tuso_1$ and the predicates $\mathsf{Tp}_i(p^{\tuso_i},x_1^{\tuso_1},\ldots,x_i^{\tuso_1})$, for all $2\le i\le n$. The axioms of $\mathsf{Tuple}_{\le n}(T)$ are
\begin{enumerate}
    \item all the axioms of $T$ relativized to $\tuso_1$;
    \item $\forall p^{\tuso_i},q^{\tuso_i},
    x_1^{\tuso_1},\ldots,x_i^{\tuso_1},
    y_1^{\tuso_1},\ldots,y_i^{\tuso_1}\;\\
    \hspace*{1cm}
    ((\mathsf{Tp}_i(p,x_1,\ldots,x_i)\land \mathsf{Tp}_i(q,y_1,\ldots,y_i))\;\to$ \\
    \hspace*{3cm} $(p=q \iff (x_1=y_1\land\ldots\land x_i=y_i)))$, \\
    for $2\le i\le n$;
    \item $\forall p^{\tuso_i}\exists x_1^{\tuso_1},\ldots,x_i^{\tuso_1}
    \;\mathsf{Tp}_i(p,x_1,\ldots,x_i)$, for $2\le i\le n$;
    \item $\forall x_1^{\tuso_1},\ldots,x_i^{\tuso_1}
    \exists p^{\tuso_i}\;\mathsf{Tp}_i(p,x_1,\ldots,x_i)$, for $2\le i\le n$.
\end{enumerate}

\begin{lemma}\label{Pairs_from_PC^2}
 $\mathsf{PC}^2(T)\rhd_1 \mathsf{Tuple}_{\le 2}(T)$.
\end{lemma}
\begin{proof}
Theory $\mathsf{PC}^2(T)$ is a theory that may be considered
to be $3$-sorted: we have the sort of elements (on which we have $T$), the sort of classes of elements, and the sort of classes that could contain either elements or other classes of elements.
We represent pairs $\langle a,b\rangle$ by Kuratowski-style pairs $\{\{a\},\{a,b\}\}$ (in the domain of classes of classes) 
and we represent elements by themselves. 
The verification of all axioms of $\mathsf{Tuple}_{\le 2}(T)$ is routine.
\end{proof}

Trivially we have:
\begin{lemma}\label{PC_to_tuple}
$\mathsf{PC}(\mathsf{Tuple}_{\le n}(T))\rhd_1 \mathsf{PC}_{\le n}(T)$.
\end{lemma}

\begin{lemma}\label{Tuple_mult}$\mathsf{Tuple}_{\le n}(\mathsf{Tuple_{\le m}}(T))\rhd_1\mathsf{Tuple}_{\le nm}(T)$.\end{lemma}
\begin{proof} In $\mathsf{Tuple}_{\le n}(\mathsf{Tuple_{\le m}}(T))$ we have $T$-domain, tuples of the elements of $T$-domain $\langle a_1,\ldots,a_k\rangle_1$, where we have $1\le k\le m$ and the tuples $\langle s_1,\ldots, s_r\rangle_2$, where $1\le r\le n$ and $s_i$'s are either element of $T$-domain or tuples $\langle a_1,\ldots,a_k\rangle_1$. Our interpretation preserves $T$-domain and all $T$ predicates. We represent a tuple $\langle a_1,\ldots,a_k\rangle$, $1\le k\le n$ as follows. We find unique $0\le r<m$ and $1\le l\le n$ such that $k=rn+l$ and put our representation to be $\langle s_1,\ldots,s_{r+1}\rangle_2$, where for $1\le i\le r$ we put $s_i=\langle a_{(i-1)r+1},\ldots,a_{(i-1)r+n}\rangle_2$ and we put $s_{r+1}=\langle a_{rn+1},\ldots,a_{rn+l}\rangle_1$.\end{proof}
From Lemmas \ref{PC_funct}, \ref{Pairs_from_PC^2}, and \ref{Tuple_mult} we get 
\begin{lemma}\label{PC_Tuple_n} $\mathsf{PC}^{2n}(T)\rhd_1 \mathsf{Tuple}_{\le 2^n}(T)$.\end{lemma}
Combining Lemmas \ref{PC_funct}, \ref{PC_to_tuple}, and \ref{PC_Tuple_n} we get
\begin{lemma}\label{PC_to_PC_le_n} $\mathsf{PC}^{2n+1}(T)\rhd_1 \mathsf{PC}_{\le 2^n}(T)$.\end{lemma}

\section{Forcing Sequentiality}

In addition to the usual kinds of interpretation we consider \emph{forcing-interpretations} 
(see the survey by Avigad \cite{avigad2004forcing} for an
overview of the method).

For a theory $T$ let $\mathsf{KM}(T)$ (Kripke models of $T$) be the following two-sorted theory. The sorts of $\mathsf{KM}(T)$ are
\begin{enumerate}
    \item $\woso$ (sort of worlds),
    \item $\doso$ (sort of elements of domains in worlds).
\end{enumerate} 
The relations $\mathsf{KM}(T)$ are
\begin{enumerate}
    \item the binary predicate $p^{\woso}\preceq q^{\woso}$ (accessibility relation on worlds), 
    \item binary predicate $D(p^{\woso},x^{\doso})$ (for a fixed $p$ it defines the domain $D_p$ of the Kripke model in the world $p$), 
    \item the predicate $R^{\star}(p^{\woso},x_1^{\doso},\ldots,x_k^{\doso})$ for each $k$-ary predicate $R$ of  the signature of $T$ (for each fixed $p$ it gives the interpretation of $R$ in the world $p$).
\end{enumerate}
For each formula $\varphi(x_1,\ldots,x_n)$ of the language of $T$, we define by recursion the formulas  $p\Vdash \varphi(x_1,\ldots,x_n)$ (the model forces $\varphi$ in the world $p$) of the language of $\mathsf{KM}(T)$:
\begin{enumerate}
    \item $p\Vdash R(x_1,\ldots,x_n)$ is $(\forall q^\woso\preceq p)(\exists r^\woso\preceq q) R^\star(r,x_1,\ldots,x_n)$;
    \item $p\Vdash \varphi(x_1,\ldots,x_n)\land \psi(x_1,\ldots,x_n)$ is\\
    \hspace*{4cm} $\big (p\Vdash \varphi(x_1,\ldots,x_n)\big)\land\big( p\Vdash \psi(x_1,\ldots,x_n)\big)$;
    \item $p\Vdash \lnot\, \varphi(x_1,\ldots,x_n)$ is $(\forall q^{\woso}\preceq p)\lnot\, \big(p\Vdash \varphi(x_1,\ldots,x_n)\big)$;
    \item $p\Vdash \forall y\, \varphi(x_1,\ldots,x_n,y)$ is $(\forall q^{\woso}\preceq p)(\forall y^{\doso})(D(q,y)\to q\Vdash \varphi(x_1,\ldots,x_n,y))$.
\end{enumerate}
The axioms of $\mathsf{KM}(T)$ are
\begin{enumerate}
    \item $\forall p^{\woso}\, p\preceq p$ \/\/ (reflexivity of $\preceq$);
    \item $\forall q^{\woso},q^{\woso},r^{\woso}((p\preceq q\land q\preceq r)\to p\preceq r)$ \/\/ (transitivity of $\preceq$);
    \item $\forall p^{\woso}\,\exists x^\doso\, D(p,x)$\/ \/ (domains are not empty);
    \item $\forall p^{\woso},q^{\woso}\,
    (q\preceq p\to \forall x^\doso(D(p,x)\to D(q,x)))$\/ \/ ($D_p\subseteq D_q$, for $q\preceq p$);
    \item $\forall p^{\woso},x_1^\doso,\ldots,x_k^\doso\,
    (R^\star(p,x_1,\ldots,x_k)\to (D(p,x_1)\land\ldots\land D(p,x_k)))$;
    \item $\forall p^{\woso},q^{\woso}\,
    (q\preceq p\to \forall x_1^\doso,\ldots,x_k^{\doso}\,(R^\star(p,x_1,\ldots,x_k)\to R^\star(q,x_1,\ldots,x_k))$\\ 
    \hspace*{1cm} 
    (downward persistence
    of the interpretations of predicates);
    \item $\forall p^{\woso}\; p\Vdash\varphi$, for all axioms $\varphi$ of $T$.
\end{enumerate}
We say that $U$ is forcing-interpretable in $T$ if there is an interpretation of $\mathsf{KM}(U)$ in $T$.

Immediately from the definition of forcing-interpretation and the fact that interpretations are closed under compositions we get
\begin{lemma}\label{weak_forcing_composition} If $T$ interprets $U$ and $U$ forcing-interprets $V$, then $T$ forcing-interprets $V$.
\end{lemma}

\begin{remark}
Although, we have not checked this carefully, it appears that it is possible to compose forcing-interpretations (and hence forcing-interpretability is a pre-order). However we don't need this fact to obtain the results of the present paper. We note that it is likely that composition of forcing-interpretations will raise the dimension of the composition.
\end{remark}

\begin{lemma}\label{PC_inside_KM} There is an 
interpretation of $\mathsf{KM}(\mathsf{PC}(T))$ in $\mathsf{PC}_{\le 2}(\mathsf{KM}(T))$.
\end{lemma}
\begin{proof}
We work in $\mathsf{PC}_{\le 2}(\mathsf{KM}(T))$ to define the desired interpretation. 

We already have an internal Kripke model $\mathcal{K}$ of $T$ inside the $\obj$-sort. That is, we have a poset of worlds $P^{\mathcal{K}}$, a family of domains $\langle D_p^{\mathcal{K}}\mid p\in P\rangle$ and interpretations $\langle R_p^{\mathcal{K}}\mid p\in P\rangle$ of all $T$ predicates $R$.

We define a Kripke model $\mathcal{S}$ of $\mathsf{PC}(T)$. The poset of worlds $P^{\mathcal{S}}$ simply coincides with $P^{\mathcal{K}}$. We call a $\clso 2$-set $A$ a \emph{name} if it consists only of pairs $\langle p, x\rangle$ such that $p\in P^{\mathcal{K}}$ and $x\in D_p$. For each world $p$ the domain $D_p^{\mathcal{S}}$ extends the domain $D_p^{\mathcal{K}}$ by all names.

Consider a world $p$.
\begin{enumerate}
    \item We put $\mathcal{S},p\Vdash S_{\obj}(x)$ iff $x\in D_p^{\mathcal{S}}$.
    \item We put $\mathcal{S},p\Vdash S_{\clso{1}}(A)$ iff $A$ is a name. 
    \item For each $k$-ary predicate $R$ of $T$ and $x_1,\ldots,x_k\in D_p^{\mathcal{S}}$ we put $\mathcal{S},p\Vdash R(x_1,\ldots,x_k)$ iff $x_1,\ldots,x_k\in D_p^{\mathcal{K}}$ and $\mathcal{K},p\Vdash R(x_1,\ldots,x_k)$. 
    \item We put $\mathcal{S},p\Vdash x\in A$ iff $x\in D_p^{\mathcal{K}}$, $A$ is a name and there exists $q\succeq p$ such that $\langle q, a\rangle \in A$.
\end{enumerate}

We note that the downward persistence of $\in$ is guaranteed
by the definition.
The fact that $\mathcal{K}$ forces the axioms of $T$ obviously implies that $\mathcal{S}$ forces the relativizations to $\obj$ of the axioms of $T$. Let us verify in a world $p$ the forceability of an instance of predicative comprehension $$\exists X^{\clso{1}}\forall x^{\obj}(x\in X\mathrel{\leftrightarrow} \varphi(x,\vec{a},\vec{A}))\text{, where $\vec{a}\in D_p^{\mathcal{K}}$ and $\vec{A}$ are names.}$$ Let $B$ be the following name:
$$B=\{ \langle q, y\rangle\mid q\preceq p, y\in D_q^{\mathcal{K}},\text{ and }\mathcal{S},q\Vdash \varphi(y,\vec{a},\vec{A})\}.$$
The definition is correct (i.e. we obtain $B$ by predicative comprehension), since $\varphi$ doesn't have
quantifiers over classes and, thus, $\mathcal{S},q\Vdash \varphi(y,\vec{a},\vec{A})$ is also expressible by a formula without quantifiers over classes. 
It is easy to see that the formula $\forall x^{\obj}(x\in X\mathrel{\leftrightarrow} \varphi(x,\vec{a},\vec{A}))$ is forced in $p$.

We did not yet treat identity of classes, but that can be easily added by setting 
$p \Vdash A= B$ iff $p \Vdash \forall z\, (z\in A \iff 
z\in B)$.
\end{proof}

\begin{corollary}\label{PC_inside_forcing}
If there is a forcing-interpretation of $U$ in $T$, then there is a forcing-interpretation of $\mathsf{PC}(U)$ in $\mathsf{PC}_{\le n}(T)$, for some $n$.
\end{corollary}
\begin{proof}In view of Lemma \ref{PC_inside_KM} it is sufficient to define an interpretation of the
theory
$\mathsf{PC}_{\le 2}(\mathsf{KM}(U))$ in $\mathsf{PC}_{\le n} (T)$. 
The latter can be done using Lemma \ref{PC_funct_mult}.
\end{proof}


For a theory $T$ we denote as $\mathsf{AS}(T)$ (Adjunctive set theory) the extension of $T$ by a fresh predicate symbol is $x\in y$ and axioms:
\begin{enumerate}
    \item $\exists x\forall y\;\lnot\, y\in x$;
    \item $\exists z\forall w\;(w\in z\mathrel{\leftrightarrow}
    (w\in x\lor w=y))$.
\end{enumerate}
A theory $T$ is called \emph{sequential}
if it admits a definitional extension to $\mathsf{AS}(T)$.

\begin{lemma}\label{forcing_AS}
Suppose $T$ is finitely axiomatized theory such that there is a one-dimensional interpretation of $T\dijosu \forall x\,(x=x)$ in $T$. Then there is a forcing-interpretation of $\mathsf{AS}(T)$ in $\mathsf{PC}_{\le n} (T)$, for sufficiently large $n$.
\end{lemma}

\begin{proof}Let $n$ be the maximum of the arities of all predicates in $T$. We have $n\ge 2$, since we have equality in the signature of $T$. We work in $\mathsf{PC}_{\le n}(T)$.
 
 A model $M$ of the signature of $T$  is a tuple consisting of a $\clso{1}$-class $D^M$ giving the domain of the model and $(k_R)$-classes $R^M$, for each $T$-predicate $R$ of arity $k_R$. Naturally, we express satisfaction of formulas inside $M$. We call $M$ a model of $T$ if all axioms of $T$ are satisfied in it. Note that here we do not require the absoluteness of equality, i.e. the equality predicate  $=^M$ is simply an equivalence relation. Note also that there is the model of $T$, whose domain is the whole $\obj$-sort and whose predicates are interpreted identically in $\mathsf{PC}_{\le n} (T)$. Using the one-dimensional interpretation of $T\dijosu \forall x(x=x)$ in $T$, for any model $M$ of $T$ we obtain a model $M'$ of $T$ such that $D^M\supsetneq D^{M'}$. 

We say that a $\clso{1}$-class $A$ is \emph{small} if there are no models $M\models T$ such that $D^{M}\subseteq A$. The class of the elements of the whole $\obj$-sort is not small, since there is the model of $T$, whose domain is the whole $\obj$-sort. Observe that, for any small $A$ with $x\not \in A$, the class $A\cup \{x\}$ is also small. Otherwise, there would be a model $M$ of $T$, whose domain is contained in  $A\cup \{x\}$, hence there would be a model $M'$ of  $T$ with $D^{M'}\subsetneq A\cup \{x\}$ and, thus, either $M'$ itself, or the result of swapping some element in its domain with $x$, would be a model $M''$ of $T$, whose domain is contained in $A$, contradicting the smallness of $A$.  

A binary relation $H$ is a pair consisting of a $\clso{1}$-class $D^{H}$ and a $\clso 2$-class $R^{H}$ such that, whenever 
$\langle x,y\rangle\in R^{H}$, we have $x,y\in D^{H}$. We use $x R^H y$ as a shorthand 
for $\langle x,y\rangle\in R^H$. We say that a binary relation $H$ end-extends a binary relation $K$ and write 
$H\supseteq_{\mathsf{end}} K$, if 
\begin{enumerate}
    \item $D^H\supseteq D^K$;
    \item for any $x,y\in D^K$ we have $xR^Ky$ iff $xR^H y$;
    \item for any $x\in D^K$ and $y\in D^{H}\setminus D^K$ we have $\lnot\, y R^H x$.
\end{enumerate}
We say that a binary relation is \emph{small} if its domain is a small $\clso{1}$-class.

To finish the proof we define an interpretation of $\mathsf{KM}(\mathsf{AS}(T))$. The poset of the worlds of the Kripke model consists of the small binary relations ordered by $\supseteq_{\mathsf{end}}$ (a small binary relation accesses all its small end-extensions).
The domain in each world is simply the whole $\obj$-sort. The interpretations of all the predicates of $T$ in all the worlds are simply the classes corresponding to the 
predicates of $T$. Finally, we interpret the predicate $\in$ in the world $H$ as $R^H$. 

It is trivial to see that, in the Kripke model thus defined, all the axioms of $T$ are forced. The forceability of the axiom of empty class $\exists x\forall y\; \lnot\, y\in x$ is clearly equivalent to the following true statement: (\dag) for any world $H$, 
there is a world $K\supseteq_{\mathsf{end}}H$  and a $\obj$-object $x$, such that, for any $\obj$-object $y$ and $L\supseteq_{\mathsf{end}}K$, 
we have $\lnot\, y R^{L} x$. The statement (\dag) 
is true since, for a given small binary relation $H$, we can take as $x$ any element outside of $D^K$, and 
define the small $K\supseteq_{\mathsf{end}} H$ with the domain 
$D^K=D^H\cup \{x\}$ so that $\lnot\, yR^K x$, for any $y\in D^K$.  
We verify the axiom of adjunction 
$\exists z\, \forall w\, (w\in z\mathrel{\leftrightarrow} (w\in x \lor w=y))$ in a similar manner. 
Thus, we indeed have defined an interpretation of  $\mathsf{KM}(\mathsf{AS}(T))$.
\end{proof}

\begin{remark}
We note that the forcing-interpretation defined in the proof
of Lemma~\ref{forcing_AS} is an analogue of what is called
an $\obj$-direct interpretation in \cite{viss:card09}.
This means that the interpretation preserves the domain and the identity relation for the $\obj$-sort. Moreover, it preserves $T$
identically on the $\obj$-sort.
\end{remark}

\section{Proof of the Main Theorem}
Recall that $\mathsf{S}^1_2$ is a weak arithmetical system capable of  the natural formalization of arguments about $P$-time computable functions (see e.g. \cite{buss:boun86}).
We will assume that finitely axiomatized theories are given
inside $\wear$ with the obvious representations of their
axiom set.

\begin{theorem}[\cite{viss:predi09}]\label{PC_and_Con} For any finitely axiomatized sequential $T$, the theory $\mathsf{PC}(T)$ interprets $\mathsf{S}^1_2+\mathsf{Con}(T)$.
\end{theorem}

Since both interpretations and forcing-interpretations lead to natural $P$-time transformations of proofs in the interpreted theory to proofs 
in the interpreting theory we have the following lemma.
\begin{lemma}\label{S^1_2_formalizable_implications} Suppose $T$ and $U$ are finitely axiomatized theories. If $T$ interprets $U$, then $\mathsf{S}^1_2\vdash \mathsf{Con}(T)\to \mathsf{Con}(U)$. If $T$ forcing-interprets $U$, then $\mathsf{S}^1_2\vdash \mathsf{Con}(T)\to \mathsf{Con}(U)$.
\end{lemma}
\begin{proof} The case of usual interpretations is well-known so we will treat only the case of forcing-interpretations. 

The forcing-interpretations correspond to polynomial-time transformations of proofs (see a discussion in \cite{avigad2004forcing,avigad2003eliminating}). This enables us to formalize in $\mathsf{S}^1_2$ the following reasoning (since $\mathsf{S}^1_2$ is able to naturally work with the polynomial transformations of strings). To prove $\mathsf{Con}(T)\to\mathsf{Con}(U)$ we assume there is a proof $P$ of contradiction from axioms of $U$ and show that then there is a proof of contradiction from axioms of $T$. Indeed, using forcing-interpretation of $U$ in $T$ we simple transform $P$ to a $T$ proof of forcability of falsity, which leads to a proof of contradiction from the axioms of $T$.\end{proof}

\begin{theorem}[G\"odel's Second Incompleteness for interpretations $\mathsf{S}^1_2$]\label{G2} No consistent $T$ interprets $\mathsf{S}^1_2+ \mathsf{Con}(T)$.\end{theorem}

Finally, we remind the reader of a basic fact about {\sf PC}.
\begin{lemma}\label{vrolijkesmurf}
Suppose $T$ is finitely axiomatized and sequential. Then, ${\sf PC}(T)$ is finitely axiomatizable.
\end{lemma}

This lemma is well known. For a proof, see e.g. \cite{viss:predi09}.

Now let us prove Theorem \ref{one_dim_thm}.
\begin{proof}
Assume for a contradiction that $T$ one-dimensionally interprets $\mathsf{PC}(T)$. We reason as follows using previously proven lemmas:
\begin{enumerate}
    \item\label{thm_step_1} $T$ one-dimensionally interprets $\mathsf{PC}^n(T)$, for any $n$ (by Lemma \ref{PC_funct});
    \item\label{thm_step_2} $T$ one-dimensionally interprets $\mathsf{PC}_{\le n}(T)$, for any $n$ (by \ref{thm_step_1}. and Lemma \ref{PC_to_PC_le_n});
    \item\label{thm_step_3-} $T$ one-dimensionally interprets $T\dijosu\forall x(x=x)$ (this trivially follows from the fact that $T$ one-dimensionally interprets $\mathsf{PC}(T)$);
    \item\label{thm_step_3} $\mathsf{PC}_{\le n}(T)$ forcing-interprets $\mathsf{AS}(T)$, for some $n$ (by \ref{thm_step_3-}. and Lemma \ref{forcing_AS});
    \item\label{thm_step_4} $\mathsf{PC}_{\le m}(\mathsf{PC}_{\le n}(T))$ forcing-interprets $\mathsf{PC}(\mathsf{AS}(T))$, for some $n$ and $m$ (by \ref{thm_step_3}. and Corollary \ref{PC_inside_forcing});
       \item\label{thm_step_4'} $\mathsf{PC}^m(\mathsf{PC}_{\le n}(T))$ forcing-interprets $\mathsf{PC}(\mathsf{AS}(T))$, for some $n$ and $m$ (by \ref{thm_step_4}. and  Lemma \ref{PC_to_PC_le_n});
    \item\label{thm_step_5} $\mathsf{PC}^n(T)$ forcing-interprets $\mathsf{PC}(\mathsf{AS}(T))$, for some $n$ (by \ref{thm_step_4'}. and Lemma \ref{weak_forcing_composition});
    \item\label{thm_step_6} $T$ forcing-interprets $\mathsf{PC}(\mathsf{AS}(T))$ (by \ref{thm_step_1}., \ref{thm_step_5}., and Lemma \ref{weak_forcing_composition});
    \item\label{thm_step_7} $\mathsf{PC}(\mathsf{AS}(T))$ interprets $\mathsf{S}^1_2+\mathsf{Con}(T)$ (by Theorem \ref{PC_and_Con});
    \item \label{thm_step_8} $\mathsf{PC}(\mathsf{AS}(T))$ interprets $\mathsf{S}^1_2+\mathsf{Con}(\mathsf{PC}(\mathsf{AS}(T)))$ (by \ref{thm_step_6}. in combination with Lemmas \ref{vrolijkesmurf} and \ref{S^1_2_formalizable_implications});
    \item \label{thm_step_9} $\mathsf{PC}(\mathsf{AS}(T))$ is inconsistent (by \ref{thm_step_8}., Theorem \ref{G2});
    \item \label{thm_step_10} $T$ is inconsistent (by \ref{thm_step_9}. and \ref{thm_step_6}.)
\end{enumerate}
So, we are done.
\end{proof}

\section{The Multi-dimensional Case}
In this section we sketch a proof of a generalization of Theorem \ref{one_dim_thm}
\begin{theorem}\label{multi_dim_thm}
No consistent finitely axiomatized $T$ can $n$-dimensionally interpret $\mathsf{PC}_{\le n}(T)$.
\end{theorem} 

Let us define the theory $T^n$. The signature of $T^n$ expands the signature of $T$ by a unary predicate $\mathsf{Dg}$ and and $n+1$-ary predicate $\mathsf{Tp}$. The axioms of $T^n$ are:
\begin{enumerate}
    \item relativization of the axioms of $T$ to $\mathsf{Dg}$;
    \item $\forall x, y_1,\ldots, y_n\,(\mathsf{Tp}(x,y_1,\ldots,y_n)\to \bigwedge\limits_{1\le i\le n} \mathsf{Dg}(y_i))$;
    \item $\forall x,x',y_1,\ldots, y_n, z_1,\ldots, z_n$\\
    \hspace*{0.3cm}
    $((\mathsf{Tp}(x,y_1,\ldots,y_n)\land \mathsf{Tp}(x,z_1,\ldots,z_n))\;\to
     (x=x' \iff
     \bigwedge\limits_{1\le i\le n}y_i=z_i))$;
    \item $\forall x\,\exists y_1,\ldots, y_n\;\mathsf{Tp}(x,y_1,\ldots,y_n)$;
    \item $\forall y_1,\ldots,y_n\,(\bigwedge\limits_{1\le i\le n} \mathsf{Dg}(y_i)\to \exists x\,\mathsf{Tp}(x,y_1,\ldots,y_n))$;
    \item $\forall x\,(\mathsf{Dg}(x)\to \mathsf{Tp}(x,\ldots,x))$.
\end{enumerate}


In $T^n$ we treat $x$ such that $\mathsf{Dg}(x)$ as individuals and we treat arbitrary objects $x$ as tuples of individuals ($x$ corresponds to the unique tuple $\langle y_1,\ldots,y_n\rangle$ such that $\mathsf{Tp}(x,y_1,\ldots,y_n)$). 

It is easy to see that the following lemma holds:
\begin{lemma}\label{T^n_lemma}There is an $n$-dimensional interpretation of $U$ in $T$ iff there is an one-dimensional interpretation of $U$ in $T^n$\end{lemma}


For a theory $T$, let  us define the theory
$\mathsf{PC}^{\mathsf{st}}(T)$. The language of $\mathsf{PC}^{\mathsf{st}}(T)$ extends the language of $T$ by a fresh unary predicate $\mathsf{Sng}$ and a binary predicate $\in$. The theory $\mathsf{PC}^{\mathsf{st}}(T)$ has the following axioms:
\begin{enumerate}
    \item the axioms of $T$ relativized to $\mathsf{Sng}$;
    \item $\forall x (\forall y(y\in x\mathrel{\leftrightarrow}y=x)\mathrel{\leftrightarrow} \mathsf{Sng}(x))$;
    \item $\forall \vec{p}\,\exists x\,\forall y\,
    (y\in x\mathrel{\leftrightarrow} (\mathsf{Sng}(y)\land \varphi(y,\vec{p}\,)))$, where $\varphi$ is a formula where all occurrences of quantifiers are of the form $\forall z\,(\mathsf{Sng}(z)\to \psi)$.
\end{enumerate}

Notice that the theory $\mathsf{PC}^{\mathsf{st}}(T^n)$ in effect is very similar to $\mathsf{PC}_{\le n}(T)$. Namely, we can simulate, in $\mathsf{PC}^{\sf st}(T^n)$, the sort $\obj$ by $x$ such that $\mathsf{Sng}(x)\land \mathsf{Dg}(x)$. We can simulate the sort $\clso k$ by arbitrary objects and we can interpret the predicate $\langle x_1,\ldots,x_k\rangle \in y$ as $\exists z\,(z\in y\land \mathsf{Sng}(z) \land \mathsf{Tp}(z,x_1,\ldots,x_k,x_1,\ldots,x_1))$. This simulation is almost an interpretation of $\mathsf{PC}_{\le n}(T)$ and the only reason why it isn't (in the sense of interpretation employed in the present paper) is that we interpret different sorts by overlapping domains.
However, in fact this doesn't matter for all the arguments in the previous parts of the paper and, by the same argument as in the proof of Theorem \ref{one_dim_thm}, we get
\begin{lemma} \label{no_PC_st_int}No consistent finitely axiomatizable theory $T$ can
one-dimensionally interpret $\mathsf{PC}^{\mathsf{st}}(T)$.
\end{lemma}

Combining Lemma \ref{no_PC_st_int} with Lemma \ref{T^n_lemma} we get
\begin{corollary} \label{no_int_PC_st_T^n}
No consistent finitely-axiomatizable $T$ can $n$-dimensionally interpret $\mathsf{PC}^{\mathsf{st}}(T^n)$.
\end{corollary} 
Since, clearly, there is a one-dimensional interpretation of the theory 
$\mathsf{PC}^{\mathsf{st}}(T^n)$ in the theory
$\mathsf{PC}_{\le n}(T)$, Corollary \ref{no_int_PC_st_T^n} implies Theorem \ref{multi_dim_thm}.


\section{Adjunctive Classes meet Adjunctive Sets}\label{forcefulsmurf}
Lemma \ref{forcing_AS} is the key part of the proof of Theorem \ref{one_dim_thm}. In this section we sketch a proof of a more general version of this result that might be interesting on its own.

Let $\mathsf{AC}_{\le n}(T)$ be the theory in the same language as $\mathsf{PC}_{\le n}(T)$. With the following axioms:
\begin{enumerate}
    \item all axioms of $T$ restricted to the domain $\obj$;
    \item $\exists X^{\clso k}\,\forall x_1,\ldots,x_k\, \lnot \, \langle x_1,\ldots,x_k\rangle \in X^{\clso k}$; 
    \item $\forall X^{\clso k}, x_1,\ldots,x_k\,
    \exists Y^{\clso k}\, \forall y_1,\ldots,y_k$\\
    \hspace*{1cm}
    $(\langle y_1,\ldots,y_k\rangle\in Y^{\clso k}\leftrightarrow 
    (\langle y_1,\ldots,y_k\rangle\in X^{\clso k}\lor \bigwedge_{1\leq i < k} y_i=x_i))$.
\end{enumerate}

Let $\mathsf{PS}_{\le  n}(T)$ be the extension of $\mathsf{AC}_{\le n}(T)$ be the following predicative separation scheme:
\begin{multline*}
\forall \vec{p}, X^{\clso 1}\, \exists Y^{\clso k}\,\forall x_1^{\obj},\ldots,x_k^{\obj}\\
(\,\langle x_1,\ldots,x_k\rangle \in Y\mathrel{\leftrightarrow}\varphi(x_1,\ldots,x_k,\vec{p}\,)
\land \bigwedge\limits_{1\le i\le k}x_i\in X\,),
\end{multline*}
where $\varphi$ is a formula such that all quantifiers in it are over the $\obj$-sort.

We define the no-universe axiom {\sf NU} as follows:
\begin{itemize}
    \item[{\sf NU}] \hspace{0.2cm}
    $\lnot\,\exists X^{\clso{1}}\,\forall x^{\obj}\; x\in X$
\end{itemize}

Inspection of the part of the proof of Lemma \ref{forcing_AS} where we defined the forcing-interpretation yields the following
sharper lemma.

\begin{lemma}\label{AS_from_non_univ} There is a forcing-interpretation of $\mathsf{AS}(T)$ in
 $\mathsf{PS}_{\le 2}(T)+{\sf NU}$.
\end{lemma}
\begin{proof}
We modify the proof of Lemma \ref{forcing_AS}. In the first part, we define in $\mathsf{PC}_{\le n}(T)$ the notion of a small $\clso{1}$-class and prove that the class of small $\clso{1}$-classes is closed under adjunctions of elements. In the second part, we use this notion of smallness to define an interpretation of $\mathsf{KM}(\mathsf{AS}(T))$. 

In the present case, the first part becomes superfluous and for the purpose of the second part we simply consider all classes to be small. Indeed one could see that the proof uses that small $\mathfrak{c}_1$-classes are closed under adjunction, that they satisfy no-universe axiom and we use predicative comprehension to form binary relations on a given small domain (this usage of comprehension could be replaced with the usage of separation). Specifically this properties of small sets are required for the verification of forceability of the axioms of empty set and adjunction.\end{proof}

\begin{remark}
It is very well possible that there is also a non-forcing-interpretation for the same result. However, it is easy to see
that we cannot generally get a non-forcing-interpretation that
preserves $T$ identically on the object sort.
\end{remark}

\begin{lemma}\label{local_interp_of_PS}
Any finite fragment of $\mathsf{PS}_{\le n}(T)$ is interpretable $\mathsf{AC}_{\le n}(T)$. 
\end{lemma}
\begin{proof}
We fix a finite fragment $U$ of $\mathsf{PS}_{\le n}(T)$. Suppose all the instances of the predicative separation present in $U$ are:
\begin{multline*}
\forall \vec{p}_i, X^{\clso 1}\, \exists Y^{\clso {k_i}}\,\forall x_1^{\obj},\ldots,x_{k_i}^{\obj} \\
(\langle x_1,\ldots,x_{k_i}\rangle \in Y \leftrightarrow
(\varphi_i(x_1,\ldots,x_{k_i},\vec{p}_i)\land \bigwedge\limits_{1\le j\le k_i}x_j\in X)),
\end{multline*}
for $i$ from $1$ to $m$. We work in $\mathsf{AC}_{\le n}$ to define the interpretation of $U$.  We take the identity interpretation for the $\obj$-domain and the signature of $T$ as well as the interpretations of $\clso k$-class domains for $k>1$. We interpret the $\clso{1}$-classes by restricting the domain. For the rest of the proof, we define this restriction.

We say that a $\clso k$-class $X$ is \emph{union friendly}, 
if for any 
$\clso k$-class $Y$, there exists a $\clso k$-class $X\cup Y$, i.e.  a $\clso k$-class $Z$ such that $$\forall x_1,\ldots,x_k\,
( \langle x_1,\ldots,x_k\rangle\in Z \leftrightarrow ( \langle x_1,\ldots,x_k\rangle\in X \lor \langle x_1,\ldots,x_k\rangle\in Y)).$$

The domain of interpretation for $\clso{1}$-classes consists of all $X$ such that for all $1\le i\le m$, parameters $\vec{p}_i$, class $B\subseteq \{1,\ldots,k_i\}$, and $\obj$-elements $z_1,\ldots,z_{k_i}$ there exists a union friendly class
\begin{equation}\label{PS_domain}Y^{\clso{k_i}}=\{\langle x_1,\ldots,x_{k_i}\rangle \mid \varphi_i(x_1,\ldots,x_{k_i}), \bigwedge\limits_{j\in \overline{B}} x_j=z_j,\text{ and }\bigwedge\limits_{j\in B} x_j\in X.\}\end{equation}

The only axiom of $U$ that is not straightforward to check is the adjunction axiom. So, in the rest of the proof, we check that, for any $\clso{1}$-class $X$ from the domain of the interpretation and $\obj$-element $x$, all the classes $X\cup\{x\}$ are in the domain of the interpretation. Indeed, we fix $1\le i\le m$, parameters $\vec{p}_i$, classes $B\subseteq \{1,\ldots,k_i\}$, and $\obj$-elements $z_1,\ldots,z_{k_i}$ and show that there exists a union friendly $\clso k$-class
$$Z^{\clso{k_i}}=\{\langle x_1,\ldots,x_{k_i}\rangle \mid \varphi_i(x_1,\ldots,x_{k_i}), \bigwedge\limits_{j\in B} x_j=z_j,\text{ and }\bigwedge\limits_{j\in B} x_j\in X\cup\{x\}\}.$$
Indeed 
{\small
$$Z^{\clso{k_i}}=\bigcup\limits_{B'\subseteq B}\{\langle x_1,\ldots,x_{k_i}\rangle \mid \varphi_i(x_1,\ldots,x_{k_i}), \bigwedge\limits_{j\in \overline{B}} x_j=z_j,\bigwedge\limits_{j\in B'} x_j\in X, \bigwedge\limits_{j\in B\setminus B'} x_j=x\}.$$
}
We observe that, since we have (\ref{PS_domain}) for $X$, all individual classes in this union exist and are union friendly. This finishes the proof since, clearly, a finite union of union friendly classes is union friendly.
\end{proof}
 
Since the interpretations constructed in Lemma \ref{local_interp_of_PS} simply restricted the $\mathfrak{c}_1$-domain, in fact they preserve the $\mathsf{NU}$-axiom and thus we have
\begin{lemma}\label{local_interp_of_PS_NU}
Any finite fragment of $\mathsf{PS}_{\le n}(T)+\mathsf{NU}$ is interpretable $\mathsf{AC}_{\le n}(T)+\mathsf{NU}$. 
\end{lemma}

\begin{corollary}\label{AS_from_non_univ_2}
There is a forcing-interpretation of ${\sf AS}(T)$ in ${\sf AC}_{\leq 2}(T)+{\sf NU}$.
\end{corollary}
\begin{proof}
By Lemma \ref{AS_from_non_univ} we have a forcing-interpretation of ${\sf AS}(T)$ in ${\sf PS}_{\leq 2}(T)+{\sf NU}$. Inspection of the 
construction shows that we need certain instances of $\mathsf{PS}$ scheme that are required to verify the forceability of adjunction and empty set axioms. In fact, these
instances do not depend on particular theory $T$. Thus $\mathsf{AS}(T)$ is interpretable in a finite fragment 
of ${\sf PS}_{\leq 2}(T)+{\sf NU}$. Hence by Lemma \ref{local_interp_of_PS_NU} we have an interpretation of ${\sf AS}(T)$ in ${\sf AC}_{\leq 2}(T)+{\sf NU}$.
\end{proof}

\begin{lemma}\label{gourmetsmurf}
There is an interpretation of ${\sf AC}_{\leq 2}(T)+{\sf NU}$ in
${\sf AS}(T)$.
\end{lemma}

\begin{proof}
It is easy to prove we can interpret ${\sf AC}_{\leq 2}$ plus
the theory of an injective binary relation {\sf InS} in {\sf AS}.
See \cite{viss:card09} for a precise definition of {\sf InS}.
Then, the interpretability of ${\sf AC}_{\leq 2}+{\sf NU}$ follows 
by the results of \cite{viss:card09}.
\end{proof}

Combining Lemma \ref{gourmetsmurf} and Corollary \ref{AS_from_non_univ_2} we get
\begin{theorem}
The theories ${\sf AC}_{\leq 2}(T)+{\sf NU}$ in
${\sf AS}(T)$ are mutually forcing-interpretable.
\end{theorem}

Inspecting the proofs, we can see that the result is even a
bit better. Both interpretations are $\obj$-direct and they
identically translate $T$ in the $\obj$-sort.

\begin{lemma}\label{PS_on_small_sets} Suppose $T$ is finitely axiomatized theory such that there is a one-dimensional interpretation of $T\dijosu \forall x\,(x=x)$ in $T$. Then, for a sufficiently large $n$, there is an interpretation of $\mathsf{PS}_{\le n}(T)+{\sf NU}$ in $\mathsf{PS}_{\le n}(T)$.
\end{lemma}
\begin{proof} As discussed in the proof of Lemma \ref{AS_from_non_univ}, the proof of Lemma \ref{forcing_AS} splits into two parts. The present Lemma is obtained by the first part of the proof. Namely we use the same definition of a small class in $\mathsf{PS}_{\le n}(T)$, although now we do not know whether
there exists the $\clso{1}$-class of all elements. None the less,
the same proof as in Lemma \ref{forcing_AS} shows that if it exists, then it isn't small. Also the same proof as before shows that small $\clso{1}$-classes are closed under adjunctions. Thus, we can interpret $\mathsf{PS}_{\le n}(T)+{\sf NU}$ 
in $\mathsf{PS}_{\le n}(T)$ by keeping everything as is, but restricting the domain of $\clso{1}$-classes to small $\clso{1}$-sets.\end{proof}

\begin{corollary} \label{from_AC_to_AS_1}
Suppose $T$ is a finitely axiomatized theory that one-dimensionally interprets $T\sqcup \forall x(x=x)$. Then, for sufficiently large $n$, the theory $\mathsf{AC}_{\le n}(T)$ forcing-interprets $\mathsf{AS}(T)$. 
\end{corollary}
\begin{proof}
Since $T$ is finitely axiomatized, the theory $\mathsf{AC}_{\le 2}(T)+ {\sf NU}$ is also 
finitely axiomatized and, hence,
by Lemma \ref{PS_on_small_sets}, the theory $\mathsf{AC}_{\le 2}(T)+ {\sf NU}$ is interpretable in $\mathsf{PS}_{\le n}(T)$, for some $n$. Since the theory $\mathsf{AC}_{\le 2}(T)+ {\sf NU}$ is finitely axiomatized, it is interpretable in a finite fragment of $\mathsf{PS}_{\le n}(T)$ and, by Lemma \ref{local_interp_of_PS}, in $\mathsf{AC}_{\le n}(T)$. By Lemma \ref{AS_from_non_univ} and Lemma \ref{weak_forcing_composition}, we get a forcing-interpretation of $\mathsf{AS}(T)$ in 
$\mathsf{AC}_{\le n}(T)$.\end{proof}

\begin{corollary}\label{from_AC_to_AS_2}
Suppose finitely axiomatizable $T\rhd_1\mathsf{AC}(T)$. Then $T$ forcing-interprets $\mathsf{AS}(T)$. 
\end{corollary}
\begin{proof} Clearly $\mathsf{AC}(T)$ interprets $T\sqcup \forall x(x=x)$. Notice that, if we replace $\mathsf{PC}_{\le n}$ with $\mathsf{AC}_{\le n}$ in all the lemmas from Section \ref{PC_and_tuples},
all the proofs work without any modifications. In particular, by the modified version of Lemma \ref{PC_to_PC_le_n}, for each $n$, the theory $\mathsf{AC}_{\le 2^n}(T)$ is interpretable in $\mathsf{AC}^{2n+1}(T)$. Thus, for each $n$, the theory $\mathsf{AC}_{\le n}(T)$ is intepretable in $T$. Hence, by Corollary \ref{from_AC_to_AS_1} and Lemma \ref{weak_forcing_composition}, the theory $T$ forcing-interprets $\mathsf{AS}(T)$.
\end{proof}

\section{Questions and Perspectives}
Our paper points to several potential directions of further research.

Despite the fact that the formulation of Theorem \ref{one_dim_thm} does not employ arithmetization, the proof reduces the result to the usual G\"odel's Second Incompleteness Theorem. Hence we have the following question:
\newcounter{prob}
\begin{enumerate}
\setcounter{enumi}{\value{prob}}
    \item Find a more direct proof of Theorem \ref{one_dim_thm} that does not employ arithmetization.
\setcounter{prob}{\value{enumi}}
\end{enumerate} 

There are questions about generalizing  Theorems \ref{one_dim_thm} and \ref{multi_dim_thm}:
\begin{enumerate}
\setcounter{enumi}{\value{prob}}
    \item Is there a finitely axiomatizable theory $T$ without finite models that does interpret $\mathsf{PC}(T)$?
    \item Is there a theory $T$ axiomatized by finitely many schemes that one-di\-mens\-ionally interprets $\mathsf{PC}^{\mathsf{schem}}(T)$?
    \item Is there a finitely axiomatizable theory $T$ without finite models that one-dimensionally interprets $\mathsf{KM}(\mathsf{PC}(T))$?
\setcounter{prob}{\value{enumi}}
\end{enumerate}
A downside of the main result of this paper is that it doesn't establish $\mathsf{PC}$ as a jump operator, since our result is applicable to finitely axiomatizable theories, but in general we do not have reasons to believe that $\mathsf{PC}(T)$ is finitely axiomatizable for all finitely axiomatizable theories $T$. Thus we have the following question:
\begin{enumerate}
\setcounter{enumi}{\value{prob}}
    \item Is it true that for any finitely axiomatized $T$ there is a finitely axiomatizable subtheory $T'$ of $\mathsf{PC}(T)$ such that $T$ doesn't one-dimensionally interpret $T'$?
\setcounter{prob}{\value{enumi}}
\end{enumerate}
Ideally, the theories $T'$ should be defined by some natural and uniform construction from $T$.

There are questions about the behaviour of $\mathsf{PC}$ operator on (interpretability) weak theories:
\begin{enumerate}
\setcounter{enumi}{\value{prob}}
    \item Characterize the interpretability degree of $\mathsf{PC}(T)$ for classical decidable theories like $\mathsf{Th}(\mathbb{N},+)$, $\mathsf{Th}(\mathbb{N},\times)$, $\mathsf{Th}(\mathbb{N},S)$, $\mathsf{Th}(\mathbb{N},<)$, $\mathsf{Th}(\mathbb{Q},<)$, $\mathsf{Th}(\mathbb{R},0,1,+,\times)$,\\ $\mathsf{Th}(\mathbb{R},0,+)$.
\setcounter{prob}{\value{enumi}}
\end{enumerate}
Also, it might be interesting to figure out the interaction of $\mathsf{PC}$ operator with various tameness notions from model theory.

Basic facts about forcing-interpretations need to be developed.
We need things like a precise definition of composition and the verification of its
desired properties.
An attractive way to do that would be to view the category of forcing-interpretations as a
co-Kleisli category. The ingredients for the desired co-monad $\mathsf{KM}$ would be the
identical one-world interpretation from ${\sf KM}(T)$ in $T$ and an interpretation of
${\sf KM}(T)$ in ${\sf KM}({\sf KM}(T))$, where worlds are interpreted as pairs of worlds.
A further issue is sameness of forcing-interpretations and the related question about the 2-category of forcing-interpretations. We can simply take over notions of sameness/isomorphism from ordinary interpretations, but we can also
think of new ones, e.g. ones inspired by bisimulations of Kripke models.

The central part of our argument is the forcing-interpretation of adjunctive set-theory. So there is a natural question, if forcing was necessary here.
\begin{enumerate}
\setcounter{enumi}{\value{prob}}
\item Is there an interpretation of $\mathsf{AS}(T)$ in $\mathsf{PS}_{\le 2}(T)+\mathsf{NU}$, for finitely axiomatizable theories?
    \item Is there always an interpretation of $\mathsf{AS}(T)$ in $\mathsf{PS}_{\le 2}(T)+\mathsf{NU}$?
    \item
    Generally, in which circumstances can forcing-interpretations be replaced with interpretations?
    In the case of finitely axiomized sequential theories or reflexive sequential theories, there is
    an argument that this can be done. However, even for arbitrary sequential theories we do not know
    whether this is always possible. 
\setcounter{prob}{\value{enumi}}
\end{enumerate}

\bibliographystyle{alpha}
\bibliography{bibliography,provint}
\end{document}